\documentclass[10pt,a4paper]{article}
\usepackage[utf8]{inputenc}
\usepackage{amsmath}
\usepackage{amsfonts}
\usepackage{amssymb}
\usepackage{amsthm}
\usepackage{mathrsfs}
\usepackage{cite}

\newtheorem{thm}{Theorem}
\newtheorem{lem}[thm]{Lemma}

\newtheorem{coro}[thm]{Corollary}
\theoremstyle{definition}

\newtheorem{example}[thm]{Example}
\newtheorem{rmk}[thm]{Remark}

\DeclareMathOperator{\sgn}{sgn}

\DeclareMathOperator{\dist}{dist}
\DeclareMathOperator{\divergence}{div}

\newcommand{\R}{\mathbb{R}}

\newcommand{\N}{\mathbb{N}}

\renewcommand{\geq}{\geqslant}
\renewcommand{\leq}{\leqslant}

\date{}
\title{Existence of differentiable curves in convex sets and the concept of direction of the flow in mass transportation}
\author{Rodolfo R\'ios-Zertuche}

\begin{document}
\maketitle
\begin{abstract}
In this paper we consider convex subsets of locally-convex topological vector spaces. Given a fixed point in such a convex subset, we show that there exists a curve completely contained in the convex subset and leaving the point in a given direction if and only if the direction vector is contained in the sequential closure of the tangent cone at that point.

We apply this result to the characterization of the existence of weakly differentiable families of probability measures on a smooth manifold and of the distributions that can arise as their derivatives. This gives us a way to consider the mass transport equation in a very general context, in which the notion of direction turns out to be given by an element of a Colombeau algebra. \end{abstract}
\tableofcontents

\section{Introduction}
\label{sec:intro}
This paper has two parts. 

The first part is the general theory, presented in Section \ref{sec:main}, that characterizes the possible derivatives of differentiable curves contained in convex subsets of locally-convex topological vector spaces. Our main motivation for the development of this general theory is its application to variational and perturbative analysis, and we will do this in upcoming papers.%

Like many very general theorems, our main abstract result, Theorem \ref{thm:characterization}, is quite simple. It states that \emph{the set of possible derivatives $\gamma'(0)$ of differentiable curves $\gamma\colon [0,+\infty)\to C$ emanating from a point $p=\gamma(0)$ in a closed, convex subset $C$ of a locally-convex topological vector field, coincides with the sequential closure of the tangent cone of $C$ at $p$;} see Section \ref{sec:characterization} for definitions and for the precise statement.  
We also give reformulations in Corollaries \ref{cor:normed} and \ref{cor:useful} that work around the issue of taking the sequential closure, and henced may be more amenable to practical use.

In the second part of the paper, we consider families of measures varying in a weakly differentiable way: we merely ask for the map $t\mapsto \int \phi\,d\mu_t$ to be differentiable at 0 for all $\phi\in C^\infty_c$. Such families arise naturally in many parts of mathematics, including the theory of the heat equation, stochastic dynamics and \textsc{spde}s, and the continuity equation in mass transportation, to name a few.  

As we show in Examples \ref{ex:deltas}, \ref{ex:uniform} and \ref{ex:explicit}, the derivative of a family of positive Radon measures is rarely a signed Radon measure, but instead in general it is a distribution, that is, a continuous real-valued functional on the space of $C^\infty$ functions. We apply our general theory to characterize in Theorem \ref{thm:families} the distributions that arise as derivatives of families of probabilities on a smooth manifold; we do this in Section \ref{sec:families}. Our result states that \emph{given a measure $\mu$, the distributions $\eta$ that arise as derivatives of differentiable families $(\mu_t)_{t\geq 0}$ starting at $\mu_0=\mu$ are precisely those that satisfy that $\langle\eta,f\rangle\geq 0$ for all nonnegative functions $f\in C^\infty_c$ that vanish on the support of $\mu$.}

In Section \ref{sec:examples} we give some important remarks and examples that should aid in understanding the content of the statement of the characterization. We also give a characterization of the existence of curves with two-sided derivatives; this is Corollary \ref{cor:two-sided}. Then in Section \ref{sec:masstransport} we consider the continuity equation from mass transportation, namely
\[\frac{d\mu_t}{dt}+\divergence(v_t\mu_t)=0,\]
in a very weak sense, and we use it to find elements of a Colombeau algebra that provide a meaning to the concept of ``direction of the flow,'' that is, we find an object that has a role similar to the classical vector field $v_t$, but in a context in which such ``movement'' can be extremely singular.

An initial version of the results presented in Section \ref{sec:families} was posted previously in the arXiv in \cite{myderivatives,myvariationalstructure}. The theory was since completely redeveloped from a different point of view until it attained the form presented here, which is also considerably simpler and more general.

\paragraph{Acknowledgements.}
The author is very grateful to Patrick Bernard, Jaime Bustillo, and Stefan Suhr  for their encouragement, support, suggestions, and numerous discussions. 

\section{Curves in convex sets}
\label{sec:main}
Consider a locally-convex topological vector space $V$ with topological dual space $V^*$. 
Let $C$ be a closed and convex subset of $V$.

We say that a subset $A$ of $V$ is a \emph{cone} if $ra$ is contained in $A$ for every $r>0$ and for every $a$ in $A$. The set $A$ is a \emph{closed cone} if it is a cone and is closed in the topology of $V$.

Let $p$ be a point in the closed, convex set $C$. The \emph{tangent cone} of $C$ at $p$ is the set of vectors $v$ such that $p+tv$ is contained in $C$ for some $t>0$. In other words, the tangent cone is the set $\mathbb R_{>0}(C-p)$. This set is a cone, but it is not necessarily closed.

Again with $p$ a point in $C$, the \emph{solid tangent cone} of $C$ at $p$ is the set of vectors $v$ in $C$ such that $\theta(v)\geq 0$ for every functional $\theta$ in $V^*$ that is nonnegative throughtout the translate $C-p$. Since the solid cone is an intersection of sets of the form $\theta^{-1}([0,+\infty))$, which are all closed, the solid tangent cone is a closed cone. It is well known that, conversely, as a consequence of the Hahn-Banach Separation Theorem, the topological closure of the tangent cone is the solid tangent cone.

An intermediate set is the \emph{sequential closure of the tangent cone} of $C$ at $p$, defined to be the set of vectors $v$ in the solid tangent cone at $p$ that can be approximated by sequences of points $\{v_i\}_{i=1}^\infty$ in the tangent cone at $p$, $v_i\to v$ as $i\to \infty$. Again, the sequential closure of the tangent cone is not necessarily topologically closed, and its closure is the solid tangent cone.

We then have

\begin{thm}\label{thm:characterization}
 Let $V$ be a locally-convex topological vector space, and let $C$ be a closed and convex subset of $V$. 
 Given a point $p$ in $C$ and a direction vector $v$ in $V$, there exists a continuous curve $c\colon[0,+\infty)\to C$ such that $c(0)=p$ and
 \[
  \left.\frac{dc(t)}{d t}\right|_{t=0+}=v
 \]
 if, and only if, $v$ is in the sequential closure of the tangent cone of $C$ at $p$.
\end{thm}
\begin{proof}
 To prove the ``if'' part, let $v_1,v_2,\dots$ be a sequence of vectors in the tangent cone to $C$ at $p$ that converges to $v$, $v_i\to v$.
 For $i=1,2,\dots$, let $\varepsilon_i>0$ be such that $p+\varepsilon_iv_i$ is contained in $C$, and also $\varepsilon_i\to0$ as $i\to +\infty$.
 Take a sequence $t_1,t_2,\dots$ of positive numbers such that $0<t_{i+1}<t_i\leq \varepsilon_i$.
 Define $p_j$ to be equal to $p+t_jv_j$. Note that $p_j\to p$ and $t_j\to 0$ as $j\to+\infty$. To define the curve $c$, first set $c(t_j)=p_j$ and then on the intervals $(t_{j+1},t_j)$ do a linear interpolation between those values . Also set $c(0)=p$; this renders $c$ continuous on $[0,t_1]$. For simplicity, we let $c(t)=p_1$ for $t\geq t_1$.
 Then $c$ is continuous, and $c(t)$ is in $C$ for all $t>0$ because of the convexity of $C$ and the fact that $p$ and $p_i$, $i=1,2,\dots$, are all contained in $C$. Also, we have
 \[\lim_{i\to+\infty}\frac{p_i-p}{t_i}=\lim_{i\to+\infty}v_j=v.\] 
 Since $c$ is a linear interpolation between these points,
 \[\lim_{t\to0+}\frac{c(t)-p}{t}=\lim_{i\to+\infty}\frac{p_i-p}{t_i}=v.\]
 Thus the curve $c$ exists.
 
 To prove the converse statement, assume the existence of $c$ as in the statement of the theorem. Let $v_i=i(c(1/i)-p)$ for $i=1,2,\dots$. Since $c(1/i)$ and $p$ are both contained in $C$, $v_i$ is in the tangent cone to $C$ at $p$. Moreover,
 \[v=\lim_{t\to0+}\frac{c(t)-p}{t}=\lim_{i\to+\infty}\frac{c(1/i)-p}{1/i}=\lim_{i\to+\infty}v_i.\]
 So $v$ is in the sequential closure of the tangent cone to $C$ at $p$.
\end{proof}

We observe the following reformulation of Theorem \ref{thm:characterization} that is appropriate for normed spaces.
\begin{coro}\label{cor:normed}
 Assume that $V$ is a normed space, that $p$ is an element of a closed, convex subset $C$ of $V$, and let $v\in V$. Then there is a continuous curve $c\colon[0,+\infty)\to C$ such that $c(0)=p$ and $dc(t)/dt|_{t=0+}=v$ if, and only if, $\theta(v)\geq 0$ for all $\theta\in V^*$ for which $\theta(p)=\inf_{q\in C}\theta(q)$.
\end{coro}
This follows easily from Theorem \ref{thm:characterization} since in this context the topological closure coincides with the sequential closure of any subset of $V$.

Here is the reformulation of the result that we will actually use in the following section.
\begin{coro}\label{cor:useful}
 Let $V$ be a locally-convex topological vector space, $p$ an element of a convex subset $C$ of $V$, and $v$ be an element of $V$. 
 Assume that $v$ and $C$ are contained in a subset $X$ of $V$ such that there is a norm on $X$ that induces the topology that $X$ inherits from $V$.
 Assume additionally that $X$ contains the intersection of the tangent cone of $C$ at $p$ with an open subset of $V$ that contains $v$.
 Then there is a curve $c\colon[0,+\infty)\to C$ such that 
 \[c(0)=p\quad \textrm{and}\quad dc(t)/dt|_{t=0+}=v\]
 if, and only if, $\theta(v)\geq 0$ for all $\theta$ in $V^*$ for which $\theta(p)=\inf_{q\in C}\theta(q)$, that is, for all continuous linear functionals $\theta$ that attain their minimum within $C$ at $p$.
\end{coro}
\begin{proof}
 We construct a sequence $v_1,v_2,\dots$ in the tangent cone to $C$ at $p$ converging to $v$, $v_i\to v$. For $i\in \N$, we take the balls $B_i$ of radius $1/i$ centered at $v$ with respect to the norm in $X$. These must all intersect that tangent cone because $v$ is in its topological closure. The tangent cone is contained in $X$, so we take $v_i$ to be any point in its intersection with $B_i$. From this point on, the rest of the proof of Theorem \ref{thm:characterization} works.
\end{proof}

\section{Families of positive measures and probabilities}
\label{sec:families}
\subsection{Characterization}
\label{sec:characterization}

The following theorem is the main result of this section and follows from Corollary \ref{cor:useful} in a way that will be explained below. Important remarks and some examples of its application can be found in Section \ref{sec:examples}. We explain the natural relation of our result with mass transportation in Section \ref{sec:masstransport}.

Denote by $C^\infty(P)$ the space of smooth functions on a smooth manifold $P$, and by $C^\infty_c(P)$ the compactly supported ones.

\begin{thm}\label{thm:families}
 Let $\mathscr M_0$ be the set of positive, Radon measures on a $C^\infty$ manifold $P$. Let $\mu$ be an element of $\mathscr M_0$ and let $\eta$ be an element of the space $\mathscr D'$ of distributions on $P$. We look for conditions for the existence of a family $(\mu_t)_{t\geq 0}$ in $\mathscr M_0$ such that
 \[\mu_0=\mu\qquad \textrm{and}\qquad \left.\frac{d}{dt}\int\phi\,d\mu_t\right|_{t=0+}=\langle\eta,\phi\rangle\]
 for all $\phi\in C^\infty_c(P)$. Such a family exists if, and only if, $\langle\eta,f\rangle\geq 0$ for all nonnegative functions $f\in C^\infty_c(P)$ that vanish on the support of $\mu$.
 
 If $\mu$ is a probability measure, the family $(\mu_t)_{t\geq 0}$ can be realized as a family of probability measures if additionally $\langle \eta,1\rangle=0$.
\end{thm}
 \begin{rmk}
 In the statement of the corollary, the equation $\langle\eta,1\rangle=0$ means that, given any locally-finite partition of unity $\{\psi_i\}_i$ of $C^\infty$ nonnegative functions $\psi_i\geq 0$ on $P$ such that $\sum_{i=1}^\infty\psi_i(x)=1$ for all $x\in P$, we have
 \[\sum_i|\langle\eta,\psi_i\rangle|<+\infty\quad\textrm{and}\quad\sum_i\langle\eta,\psi_i\rangle=0.\]
 \end{rmk}
 \begin{rmk}
  The same corollary holds mutatis mutandis for compactly supported $\mu$ and $\eta$ with $f,\phi\in C^\infty(P)$.
 \end{rmk}

 We take the space $C^\infty(P)$ (sometimes denoted $\mathscr E$) of smooth functions on $P$ to be endowed with the topology of convergence on compact sets, which is induced by the family of seminorms $|\cdot|_{K,k}$ where $K\subset P$ is a compact set, $k\geq 0$ is an integer, and 
 \begin{equation}\label{eq:seminorms}
  |f|_{K,k}=\sum_{|I|\leq k}\sup_{p\in K}|\partial^If(p)|,\quad f\in C^\infty(P).
 \end{equation}
 The subspace $C^\infty_c(P)\subseteq C^\infty(P)$ (sometimes denoted $\mathscr D$) of compactly supported smooth functions inherits the topology from $C^\infty(P)$.
 
 Denote by $\mathscr D'$ the space of distributions on $P$, and by $\mathscr E'$ the space of compactly-supported distributions on $P$. These are endowed with the weak* topology with respect to $C^\infty_c(P)$ and $C^\infty(P)$, respectively. We further write $\mathscr E'=\cup_n\mathscr E'^n$, where $\mathscr E'^n$ is the set of compactly-supported distributions of order $\leq n$ for $n\in\N\cup\{0\}$.
 
   For a positive integer $m$ and a compact set $K\subset P$, let $|\cdot|_{K,m}$ be the norm given in \eqref{eq:seminorms}, and let $|\cdot|'_{K,m}$ be the dual seminorm given by
 \begin{equation}\label{eq:dualseminorms}
  |\xi|'_{K,m}=\sup_{|g|_{K,m}\leq1} |\langle g,\xi\rangle|,\qquad\xi\in\mathscr E',
 \end{equation}
 where the supremum is taken over all functions $g\in C^\infty_c(P)$ such that $|g|_{K,m}\leq1$. 
 
 For the proof of Theorem \ref{thm:families} we will need this lemma, which is proved below.
\begin{lem}\label{lem:boundednormability}
 Let $n$ be a positive integer, $\gamma>0$, and $U$ a bounded open subset of $P$, so that its closure $\overline U$ is compact. Consider the subset $X$ of $\mathscr E'^n$ consisting of distributions $\xi$ of degree at most $n$ with support in $U$ and such that $|\xi|'_{\overline U,n}\leq \gamma$. Then the seminorm $|\cdot|'_{\overline U,n}$ is nondegenerate on $X$, and within $X$ it acts as a norm that induces the topology that $X$ inherits from $\mathscr E'^n$, that is, the weak* topology with respect to $C^\infty(P)$. In particular, the sequential closure of subsets of $X$ coincides with their closure.
\end{lem}

\begin{proof}[Proof of Theorem \ref{thm:families}]
 By means of a partition of unity, we may assume at the outset that both $\mu$ and $\eta$ are compactly supported. 
 
 Let us explain how we will apply Corollary \ref{cor:useful}.
 Let $n>0$ be the degree of $\eta$. Set $V=\mathscr E'$, $C=\mathscr M_0$, $p=\mu$, and $v=\eta$, and from the corollary we will get $c(t)=\mu_t$. 
 Fix a bounded, open subset $U$ of $P$ that contains the supports of $\mu$ and $\eta$, and let $\gamma=2|\eta|'_{\overline U,n}$
 Let the set $X\subset \mathscr E'^n$ be the set of distributions $\xi$ of order at most $n$ whose supports are contained in $U$ and such that $|\xi|'_{\overline U,n}\leq 2|\eta|'_{\overline U,n}$.
 From Lemma \ref{lem:boundednormability}, we know that the topology on $X$ is induced by the seminorm $|\cdot|'_{\overline U,n}$, which is nondegenerate within $X$.
 We also observe that $V$ is locally convex. 
 The set $X$ contains all signed Radon measures $\nu$ with $|\nu|_{\overline U,n}'<2|\eta|'_{\overline U,n}$, so in particular it contains the intersection of the tangent cone of $C$ at $p$ with the set of distributions $\xi$ with $|\xi|_{\overline U,n}'<2|\eta|'_{\overline U,n}$. This set of distributions is an open set that contains $v$.
 So we can apply Corollary \ref{cor:useful}, which is then equivalent to the following assertion: A family $(\mu_t)_{t\geq 0}$ as in the statement of Theorem \ref{thm:families} exists if, and only if, $\langle\eta,f\rangle\geq0$ for all functions $f\in C^\infty_c(P)$ such that 
 \begin{equation}\label{eq:minimum}
  \int f\,d\mu\leq \int f\,d\nu\quad \textrm{for all}\quad \nu\in\mathscr M_0.
 \end{equation}

 To finish the proof of the first part of Theorem \ref{thm:families}, we just need to understand these functions $f$.
 Let $f\in C^\infty_c(P)$ satisfy \eqref{eq:minimum}. Clearly, the support of $\mu$ must be contained in the set of points $p$ such that $f(p)=\inf_{q\in P}f(q)$, so in order to prove the theorem we just need to argue that this infimum equals zero, or equivalently, that $\int f\,d\mu=0$. Taking $\nu_a=a\mu\in\mathscr M_0$ for $a>0$, it is easy to see that if $\inf_{q\in P}f(q)\neq 0$, then $a$ can be taken appropriately so that $\int f\,d\nu_a=a\int f\,d\mu<\int f\,d\mu$, thus contradicting \eqref{eq:minimum}.
 
 If $\mu$ is a probability and $\langle\eta,1\rangle=1$, we may divide each measure $\mu_t$ in the family by its total mass to get probability measures $\mu_t/\mu_t(P)$; the condition $\langle\eta,1\rangle=1$ ensures that this procedure does not break the differentiability of the normalized family at $t=0+$.
\end{proof}

\begin{proof}[Proof of Lemma \ref{lem:boundednormability}]
 Let us recall that the \emph{strong topology} on $\mathscr E'$ is the one induced by the seminorms $|\cdot|'_{K,m}$ defined in equation \eqref{eq:dualseminorms}.

 Recall that $\mathscr E'$ is a Montel space. A set $A$ in $\mathscr E'$ is, by definition, \emph{bounded} if for each $f\in C^\infty(P)$ there is a constant $C_f>0$ such that $\langle\xi,f\rangle\leq C_f$ for all $\xi$ in $A$. Thus $X$ is bounded. Proposition 34.5 in \cite{treves} tells us that, in Montel spaces, the strong and weak* topologies coincide in bounded sets, so this is the case in $X$. It is easy to see that, within $X$, the seminorm $|\cdot|'_{\overline U,n}$ dominates all other seminorms of the form $|\cdot|'_{K,k}$ with $K$ compact and $k>0$. Thus, within the set $X$ the seminorm $|\cdot|'_{\overline U,n}$ induces the weak* topology. Its nondegeneracy is easy to check as well.
\end{proof}

 \subsection{Examples and remarks}
\label{sec:examples}

\begin{example}\label{ex:deltas}
 We show that the derivative of a family of measures $(\mu_t)_t$ \emph{need not itself be a measure}. Let $P=\R$ and $\mu_t=\delta_t$, the Dirac delta at $t\in\R$. Then for all $f\in C^\infty(P)$ we have
 \[\left.\frac{d}{dt}\right|_{t=0}\int f\,d\mu_t=\left.\frac{d}{dt}\right|_{t=0}f(t)=f'(0),\]
 so that the derivative $\eta=d\mu_t/dt|_{t=0}$ is the distribution $\eta=-\partial_x \delta_0$, where $\partial_x$ indicates the derivative in the domain $\R=P$ of the measures, and is taken here in the sense of distributions.
\end{example}

\begin{rmk}\label{rmk:notdistributional}
 The derivative $d\mu_t/dt|_{t=0+}$ that we are interested in \emph{does not coincide with the usual distributional derivative}. In order to take the derivative in the sense of distributions, we would need to consider $(\mu_t)_{t\geq0}$ as a distribution on the set $P\times [0,+\infty)$, and we do not do that. 
 
 Let us consider the situation in Example \ref{ex:deltas} as a purely distributional derivative: we take $\mu_t=\delta_t$ on $P=\R$ and, since it is well defined and smooth for all $t$ in $\R$, we think of the entire family $\mu_t$ as defining a distribution $\nu$ on $\R^2$, acting on functions $f\in C^\infty(\R^2)$ by integration:
 \begin{multline*}
  \langle\nu,f\rangle=\int_{-\infty}^{+\infty}\int_{-\infty}^{+\infty}f(x,t)d\mu(x)dt= \\
  \int_{-\infty}^{+\infty}\int_{-\infty}^{+\infty}f(x,t)d\delta_t(x)dt=\int_{-\infty}^{+\infty}f(t,t)dt.
 \end{multline*}
 Then the distribution $\partial_t\nu$ (where the derivative is taken in the sense of distributions) acts like this on $f\in C^\infty_c(\R^2)$:
 \[\langle\partial_t\nu,f\rangle=-\langle\nu,\partial_t f\rangle=-\int_{-\infty}^{+\infty}\frac{\partial f}{\partial t}(s,s)\,ds.\]
 Clearly, there is no reason for this to coincide with the result we obtain in Example \ref{ex:deltas}, which for $f\in C^\infty_c(\R^2)$ perhaps amounts to $\partial f/\partial x|_{(x,t)=(0,0)}$. The important point, in any case, is that if we considered the derivative with respect to $t$ as a distributional one, we would need to be taking test functions that depended on that variable, and we do not do this; note that in Example \ref{ex:deltas} the function $f$ depends only on the real variable corresponding to $P$ and not on $t$.
\end{rmk}

\begin{example}\label{ex:deformationsdelta}
 Let again $P=\R$, and let us use Theorem \ref{thm:families} to determine the ways in which the Dirac delta at 0, $\delta_0$, can be deformed. The condition in Theorem \ref{thm:families} requires us to look for distributions $\eta$ such that $\langle\eta,f\rangle\geq 0$ for all $f\in C^\infty_c(\R)$ such that $f(x)\geq 0$ for all $x\in \R$ and $f(0)=0$. Since the Taylor expansion of any such function $f$ starts at the quadratic term, which must have a nonnegative coefficient, we conclude that $\eta$ must be of the form 
 \[\eta=\nu+a\,\partial_x\delta_0+b\,\partial_x^2\delta_0\]
 for a positive measure $\nu$ on $\R$, a real number $a$, and a positive number $b>0$.
\end{example}

\begin{example}\label{ex:probdeformationsdelta}
 Going back to the situation of Example \ref{ex:deformationsdelta}, we now use the second part of Theorem \ref{thm:families} to determine the ways in which the distribution $\delta_0$ can be deformed within the space of probability measures. In this case, we have the additional requirement that the derivative $\eta$ satisfy $\langle\eta,1\rangle=0$. The reader will immediately recognize that this amounts to the requirement that the measure $\nu$ vanish everywhere or, in other words, that $\eta$ be of the form
 \[\eta=a\,\partial_x\delta_0+b\,\partial_x^2\delta_0\]
 for $a\in \R$, $b>0$.
\end{example}

\begin{rmk} 
 Although the statement of Theorem \ref{thm:families} deals only with one-sided families $(\mu_t)_{t\geq0}$, it is clear what the requirement is on the distribution $\eta$ for the existence of a two-sided family $(\mu_t)_{t\in\R}$ with derivative $\eta$: $\eta$ and $-\eta$ must satisfy the conditions in Theorem \ref{thm:families}. For the reader's convenience, we state this in full:
 \begin{coro}\label{cor:two-sided}
  Let $\mathscr M_0$ be the set of positive, Radon measures on a $C^\infty$ manifold $P$. Let $\mu$ be an element of $\mathscr M_0$ and let $\eta$ be an element of the space $\mathscr D'$ of distributions on $P$. We look for conditions for the existence of a (two-sided) family $(\mu_t)_{t\in \R}$ in $\mathscr M_0$ such that $\mu_0=\mu$ and $\frac{d}{dt}\int \phi\,d\mu_t|_{t=0}=\langle\eta,\phi\rangle$ for all $\phi\in C^\infty_c(P)$. Such a family exists if, and only if, $\langle\eta,f\rangle=0$ for all nonnegative functions $f\in C^\infty_c(P)$ that vanish in the support of $\mu$.
  
  If $\mu$ is a probability measure, the family $(\mu_t)_{t\in\R}$ can be realized as a family of probability measures if additionally $\langle\eta,1\rangle=0$.
 \end{coro}
\end{rmk}

\begin{example}\label{ex:two-sided-delta}
 Going back to the situation of Examples \ref{ex:deformationsdelta} and \ref{ex:probdeformationsdelta}, we now determine the possible two-sided deformations of $\delta_0$ both as two-sided families of positive, Radon measures and as two-sided families of probability distributions on $P$. From Corollary \ref{cor:two-sided}, we see that the condition that needs to be satisfied by the derivative $\eta$ is that $\langle\eta,f\rangle=0$ for all $f\in C^\infty_c(\R)$ such that $f\geq 0$ throughout $\R$ and $f(0)=0$. In this case, the quadratic term of the Taylor expansion of $f$ impedes $\eta$ from having a term of order 2, so for the case of arbitrary positive measures we get that $\eta$ must be of the form
 \[\eta=\nu+a\,\partial_x\delta_0,
 \]
 for a positive measure $\nu$ and a real number $a$. Similarly, in the case of families of probability measures the additional condition that $\langle\eta,1\rangle=0$ forces $\nu$ to vanish and $\eta$ to be of the form
 \[\eta=a\,\partial_x\,\delta_0,\quad a\in\R.\]
\end{example}

\begin{example}\label{ex:uniform}
 \emph{Distributions of any degree may arise as derivatives of families of measures.} Consider the case in which $\mu$ is the uniform probability measure on a bounded, open set $U\subseteq P$. Then any distribution $\eta$ supported in $\overline U$ satisfies the condition of the first part of either Theorem \ref{thm:families} or Corollary \ref{cor:two-sided}, and any distribution $\eta$ additionally satisfying $\langle\eta,1\rangle=0$ will be good for the second part of the statements of therein contained.
\end{example}

\begin{example}\label{ex:explicit}
 We work out Example \ref{ex:uniform} explicitly in the special case of $P=\R$, $U=(-\frac12,\frac12)$, $\mu=\chi_{U}dx$ the uniform distribution on $U$, $k\in\N$, $\eta=(-1)^k\partial^k\delta_0$. Let $\psi\colon\R\to[0,1]$ be a compactly-supported, $C^\infty$ function that is equal to 1 in a neighborhood of 0, satisfies $\psi(x)=\psi(-x)$, $\int_\R\psi(x)dx=1$, and is decreasing on $[0,+\infty)$. Pick some $0<q<1$. For $t\in\R$ with $|t|$ small enough, we let $\mu_t$ be the measure that corresponds to the density $\rho(x,t)\,dx$ with
 \[\rho(x,t)=\chi_U(x)+(-1)^k(\sgn t)|t|^q\frac{d^k\psi}{dx^k}\left(|t|^{\frac{q-1}{k+1}}x\right), \quad t\neq 0,\]
 and $\rho(x,0)=\chi_U(x)$. We need $|t|$ to be small enough that $\rho(x,t)$ will be nonnegative for all $x$; this defines $\mu_t$ only for small $|t|$, yet its smooth extension to all of $t\in\R\setminus\{0\}$ exists.  We remark that at $t=0$ and $x=0$, the function $\rho$ is not differentiable with respect to $t$ because the slope is vertical. 
 
 Let us check that this family $(\mu_t)_{t\in\R}$ works, that is, that its derivative is indeed equal to $\eta$. Let $f\in C^\infty(\R)$; then we have 
 \begin{align*}
 \frac{d}{dt}\int f\,d\mu_t&=\frac{d}{dt}\int_\R f(x)  \rho(x,t)dx\\
  &=\frac{d}{dt}\int f(x)\left(\chi_U(x)+(-1)^k\sgn t|t|^q\frac{d^k\psi}{dx^k}(|t|^{\frac{q-1}{k+1}}x)\right)dx \\
  &=\frac{d}{dt} \sgn t|t|^\frac{q+k}{k+1}\int \frac{d^kf}{dx^k}(x)\psi(|t|^\frac{q-1}{k+1}x)dx
 \end{align*}
 after $k$ integrations by parts. In order to simplify the notation, let $h=d^kf/dx^k$. We take the derivative and we get
 \begin{align*}
  &(\sgn t)|t|^\frac{q-1}{k+1}\left(\tfrac{q+k}{k+1}  \int h(x)\psi(|t|^\frac{q-1}{k+1}x)dx +\tfrac{q-1}{k+1}|t|^\frac{q-1}{k+1}\int xh(x)\psi'(t^\frac{q-1}{k+1}x)dx\right)\\
  &=(\sgn t)|t|^\frac{q-1}{k+1}\left(\tfrac{q+k}{k+1}  \int h(x)\psi(|t|^\frac{q-1}{k+1}x)dx-\tfrac{q-1}{k+1}\int (xh(x))'\psi(|t|^\frac{q-1}{k+1}x)dx\right)\\
  &=(\sgn t)|t|^\frac{q-1}{k+1}\left(\int h(x)\psi(|t|^\frac{q-1}{k+1}x)dx-\tfrac{q-1}{k+1}\int xh'(x)\psi(|t|^\frac{q-1}{k+1}x)dx\right).
 \end{align*}
 This is easily seen to tend to $h(0)-\frac{q-1}{k+1}0\,h'(0)=h(0)=\langle\eta,f\rangle$ as $t\to0$, which is what we wanted.
\end{example}

\begin{rmk}\label{rmk:support}
 The conditions in Corollary \ref{cor:two-sided} force the support of $\eta$ to be contained inside the support of the measure $\mu$. However, take a look at the following example. 
\end{rmk}

\begin{example}[A considerably ugly situation]\label{ex:singcont}
 For the conditions of Theorem \ref{thm:families} or Corollary \ref{cor:two-sided} to be satisfied, it is \emph{not} true that the support of the distribution $\eta$ must be contained inside the set of points that hold the measure of $\mu$. (Note that the support of $\mu$ is the closure of this set.)  Immitating the construction of the Cantor denisty, it is easy to construct a singular-continuous measure supported on the interval $[0,1]$ whose measure will be contained in a set of Lebesgue measure zero; for instance, one can recursively subdivide the interval in halves and assign 1/3 of the measure to the rationals in the lower half and 2/3 to the rationals in the upper half, and iterating for the halves of each of these intervals, while declaring that the set of irrational numbers in $[0,1]$ has $\mu$ measure zero. In this case, Theorem \ref{thm:families} and Corollary \ref{cor:two-sided} show that the possible deformations are the same as for the uniform measure on $[0,1]$; see Example \ref{ex:uniform}. Many of the distributions thus arising are supported in sets that are disjoint from the set holding the measure; as an example, take $\eta=\partial^k_x\delta_{\sqrt2/2}$, for any positive integer $k$. 
 
 We additionally remark that the existence of ugly examples like these is what forces us to the rather abstract treatment of the theory that we have resorted to in proving the results of Section \ref{sec:characterization} rather than a more constructive and explicit technique. 
\end{example}

\begin{rmk} \label{rmk:caratheodory}
 Here, we want to establish the relation between our result and the Carath\'eodory positivity criterion for holomorphic functions \cite{caratheodory}.
 Carath\'eodory proved that a necessary and sufficient condition for the signed measure $\mu$ on the interval $[0,2\pi)$ to be positive is for its Fourier coefficients
  \[
  a_n=\int_0^{2\pi}e^{-in\theta}d\mu(\theta)
 \]
 to satisfy 
 \begin{equation}\label{eq:caratheodory}
  \sum_m\sum_na_{m-n}\lambda_m\overline{\lambda_n}\geq 0 
 \end{equation}
 for any complex numbers $\lambda_0,\lambda_1,\dots,\lambda_N$.
 
 Now consider the case in which we have a family of positive measures $(\mu_t)_{t\geq 0}$ with $d\mu_t/dt|_{t=0+}=\eta$, as in Theorem \ref{thm:families}. Then we have Fourier coefficients $a_n(t)=\int_0^{2\pi}e^{-in\theta}d\mu_t(\theta)$ defined for each $t\geq 0$, and the Fourier coefficients of $\eta$ will be given by $\langle\eta,e^{-in\theta}\rangle=a_n'(0)$. Since each $\mu_t$ is positive, the criterion \eqref{eq:caratheodory} must hold for each $t\geq 0$. Thus the condition in Theorem \ref{thm:families} is equivalent to requiring that the Fourier coefficients $a_n'(0)$ of $\eta$ satisfy
 \[\sum_m\sum_na'_{m-n}(0)\lambda_m\overline{\lambda_n}\geq 0\]
 whenever $\lambda_0,\lambda_1,\dots,\lambda_N$ are complex numbers such that
 \[\sum_m\sum_na_{m-n}(0)\lambda_m\overline{\lambda_n}=0. \]
\end{rmk}

\subsection{Mass transportation}
\label{sec:masstransport}
\paragraph{Motivation.} For simplicity let $P=\R^n$, $n\geq 1$. 
Recall that for $p> 1$ the Wasserstein metric between two probabilities $\mu$ and $\nu$ on $\R^n$ with finite $p^\textrm{th}$ momenta on $\R^n$ can be defined as
\[W_p(\mu,\nu)=\left(\inf_{\gamma\in\Gamma(\mu,\nu)}\int_{\R^n\times \R^n}\dist(x,y)^p\,d\gamma(x,y)\right)^{\frac1p},\]
where the infimum is taken over all measures $\gamma$ on $\R^n\times \R^n$ whose marginals are $\mu$ and $\nu$.
When $(\mu_t)_{t\in\R}$ is a family of smooth probability densities that defines an absolutely continuous curve in the space of probabilities endowed with the metric $W_p$, it has been shown \cite[Chapter 8]{ambrosiogiglisavare} that the derivative of $(\mu_t)_t$ exists for almost every $t$, and can be interpreted as the divergence of a vector field, that is, there are vector fields $v_t\colon \R^n\to\R^n$ that satisfy the \emph{continuity equation},
\begin{equation}\label{eq:continuityeq}
 \frac{d\mu_t}{dt}+\divergence(\mu_tv_t)=0, 
\end{equation}
for almost every $t\in\R$.
The interpretation, which follows from Gau\ss's theorem, is that the mass of $\mu_t$ is being carried or transported by the flow of the vector field $v_t$. This gives a way to assign a vector field $v_t$ to the distribution $d\mu_t/dt$, and this vector field gives a notion of ``direction of the movement.'' The vector field $v_t$ is not unique: it is ambiguous up to addition of a vector field $u_t$ such that $\divergence(\mu_tu_t)=0$ for all $t\in\R$. Since %
the set of possible $u_t$ is a closed subspace of $L^p(\mu_t)$, one can choose the vector field $v_t$ to be the minimizer of the $L^p(\mu_t)$ norm for each $t$. It has been shown \cite{ambrosiogiglisavare} that the minimizer is in fact a gradient vector field, $v_t=\nabla\phi_t$ for some functions $\phi_t\colon\R^n\to\R$.

In the theory of $W_p$-absolutely continuous families of measures, $v_t$ exists only for almost every $t$. Consider the situation of Example \ref{ex:explicit}: the family $\mu_t=\rho(x,t)dx$ given in the example is $W_p$-absolutely continuous, but $v_0$ is not defined for $k\geq 1$. 
The results of Section \ref{sec:characterization} indicate that in cases like this, %
and surely also in the cases of families of measures that are not $W_p$-absolutely continuous but are still weakly differentiable,
the distributions that can arise as their derivatives can be much more general than those of the form $\divergence(\mu_tv_t)$, for some vector field $v_t$; cf. Example \ref{ex:uniform}. As it turns out, an arbitrary distribution may $\eta$ may arise as the derivative of a family of measures, $d\mu_t/dt|_{t=t_0}=\eta$, and we can try to use the continuity equation \eqref{eq:continuityeq} to try to assign an object that will give an idea of direction of the movement determined by $\eta$.

The assignment of an object conveying the direction of movement and analogous to $v_t$ can be done using Colombeau algebras. These algebras were developed \cite{colombeau} to provide a context in which distributions could be multiplied. The spaces of distributions are subsets of these algebras. Roughly speaking, the solution to the multiplication problem is to record, instead of the distribution itself, all possible smoothings of the distribution, because there is no difficulty in multiplying smooth densities. As we will explain below, an equivalence relation is then introduced on a certain set of families of smooth functions, and its equivalence classes are the elements of the algebra.

\paragraph{Construction.}
To define the relevant Colombeau algebra, we follow \cite[Section 8.5]{colombeaubook}. Let $\mathcal E(\R^n)$ be the set of families $(f_\varepsilon)_{0<\varepsilon<1}$ of functions $f_\varepsilon\in C^\infty(\R^n)$ indexed by $0<\varepsilon<1$, such that for each compact set $K\subset \R^n$ and every multi-index $I$ there are $N\in \N$, $\gamma>0$, and $c>0$ such that
\[\sup_{x\in K}\left|\partial^If_\varepsilon(x)\right|\leq \frac{c}{\varepsilon^N}\quad\textrm{if $0<\varepsilon<\gamma$}.\]
We define the ideal $\mathcal N(\R^n)$ of $\mathcal E(\R^n)$ to be the set of families $(f_\varepsilon)_{0<\varepsilon<1}$ such that for all compact sets $K$, for all multi-indices $I$, and for all $q\in \N$ there exist $c>0$ and $\gamma>0$ such that 
\[\sup_{x\in K}\left|\partial^If_\varepsilon(x)\right|\leq c\varepsilon^q\quad\textrm{if $0<\varepsilon<\gamma$.}\]
This means that the elements of $\mathcal N(\R^n)$ have a fast decay (faster than any power of $\varepsilon$) as $\varepsilon \searrow0$. The \emph{Colombeau algebra} is the quotient
\[\mathcal G(\R^n)=\mathcal E(\R^n)/\mathcal N(\R^n).\]
All distributions $\eta$ are represented in $\mathcal G(\R^n)$ because the families of smoothings by convolution with a mollifier are contained there. We denote by $[\eta]$ the set of elements of of $\mathcal G(\R^n)$ that are associated to the distribution $\eta$. That is, $(f_\varepsilon)_\varepsilon$ belongs to the subset $[\eta]$ of $\mathcal G(\R^n)$ if, for all $\phi\in C^\infty_c(\R^n)$,
\[\lim_{\varepsilon\to0+}\int\phi(x)f_\varepsilon(x)\,dx=\langle\eta,\phi\rangle.\]

We now proceed to use the Colombeau algebra $\mathcal G(\R^n)$ to define the object corresponding to the direction of the motion. 
Let $\mu$ be a probability measure on $\R^n$ and let $\eta$ be a distribution satisfying the conditions in Corollary \ref{cor:two-sided} including the last part for families of probability measures, so that there exists a family of probabilities $(\mu_t)_{t\in\R}$ with derivative $\eta$ at 0 and $\mu_0=\mu$.

Let $[\mu]$ be the subset of  $\mathcal G(\R^n)$ associated to $\mu$. Note that $[\mu]$ always contains representatives $(f_\varepsilon)_\varepsilon\in[\mu]$ satisfying 
\begin{equation}\label{eq:regularcolombeau}
\int_{\R^n}f_\varepsilon(x)\,dx=1\quad\textrm{and}\quad f_\varepsilon\geq 0\quad \textrm{for $0<\varepsilon<1$}
\end{equation}
because these can be obtained by convolution of $\mu$ with a mollifier of mass 1.
Let $\mathcal V(\mu)$ be the set of families $(v^\varepsilon)_{0<\varepsilon<1}$ of smooth vector fields such that, for all representatives satisfying \eqref{eq:regularcolombeau}, we have
\[(\divergence(f_\varepsilon v^\varepsilon))_{0<\varepsilon<1}\in \mathcal E(\R^n).\]
Note that the set $\mathcal V(\mu)$ is a vector space. 

Choose a representative $(g_\varepsilon)_{0<\varepsilon<1}$ in the subset $[\eta]$ that additionally satisfies that the support of $g_\varepsilon$ is contained in the support of $f_\varepsilon$; by Remark \ref{rmk:support}, 
such a representative can always be constructed by convolution.

Consider the continuity equation
\begin{equation}\label{eq:continuitysmooth}
 g_\varepsilon+\divergence(f_\varepsilon v^\varepsilon)=0. 
\end{equation}
The space $\mathcal V(\mu)$ always contains a solution $(v^\varepsilon)_{0<\varepsilon<1}$ to this equation; we know this because for each $\varepsilon$ this is just the classical smooth case, treated in \cite[Chapter 8]{ambrosiogiglisavare}. Accordingly, we may choose $v^\varepsilon$ to be minimal with respect to the $L^2$ norm, and it will correspond to a gradient vector field $\nabla\phi_\varepsilon$ for some functions $\phi_\varepsilon\colon \R^n\to\R$, $0<\varepsilon<1$.

The (non-unique) family $(v^\varepsilon)_{0<\varepsilon<1}\in \mathcal V(\mu)$ is the object that we have been pursuing, as it gives precise meaning to the notion of ``direction of the movement'' of $\mu$ when the change of the distribution of mass is given by $\eta$.

\paragraph{Remarks and examples.}
In \cite{ambrosiogiglisavare}, the authors make the choice of taking the tangent space to the set of probabilities to be the set of vector fields that arise as $v_t$ in the continuity equation. What tangent space to take is a question that depends a bit on the intended application and a bit on personal taste. Our description above suggests an alternative in the form of a subset of the corresponding Colombeau algebra, and Theorem \ref{thm:families} suggests yet another alternative.

It would certainly be interesting to know which parts of the theory of mass transport still hold in the context of weakly-differentiable families of probabilities, a question that we leave for later research.

The following examples are intended to clarify the construction above and its interpretation.

\begin{example}\label{ex:massdeltas}
 We go back to the situation of Example \ref{ex:deltas} and we compute a representative $(v^\varepsilon)_\varepsilon$ of the direction of movement object.
 Let $P=\R$, $\mu_t=\delta_t$, $\mu=\mu_0=\delta_0$, $\eta=d\mu_t/dt|_{t=0}=-\partial_x\delta_0$. Let $\psi\colon\R\to[0,1]$ be a compactly-supported, $C^\infty$ function that is equal to 1 in a neighborhood of 0, and satisfies $\psi(x)=\psi(-x)$ and $\int_\R\psi(x)dx=1$. Also let $\psi_\varepsilon(x)=\frac1\varepsilon\psi(x/\varepsilon)$, $f_\varepsilon=\psi_\varepsilon*\mu=\psi_\varepsilon$, and $g_\varepsilon=\psi_\varepsilon*\eta=-\psi_\varepsilon'$, where the last equality is true because for all test functions $\phi\in C^\infty_c(\R)$, we have
 \begin{multline}\label{eq:convolutionderivative}
  \langle g_\varepsilon,\phi\rangle=\langle\eta,\psi_\varepsilon*\phi\rangle=\langle-\partial_x\delta_0,\psi_\varepsilon*\phi\rangle=\langle\delta_0,\partial_x(\psi_\varepsilon*\phi)\rangle\\
  =\langle\delta_0,(\partial_x\psi_\varepsilon)*\phi\rangle=(\partial_x\psi_\varepsilon)*\phi(0)=\langle- \partial_x\psi_\varepsilon,\phi\rangle.
 \end{multline}
 Also, in dimension 1 the divergence is simply the derivative.
 Thus in this case the smoothed version  \eqref{eq:continuitysmooth} of the continuity equation \eqref{eq:continuityeq} becomes
 \[(-\psi_\varepsilon')+(\psi_\varepsilon v^\varepsilon)'=0.\]
 Taking $v^\varepsilon\equiv1$ solves this equation. It also makes sense: this can be interpreted as movement to the right, which is exactly what the mass of the family $(\delta_t)_{t\in\R}$ is doing.
\end{example}

\begin{example}\label{ex:massexplicit}
 We work out explicitly the same situation as discussed in Example \ref{ex:explicit}, namely, $P=\R$, $U=(-\frac12,\frac12)$, $\mu=\chi_Udx$ the uniform distribution on $U$, $k\in \N$, $\eta=(-1)^k\partial^k\delta_0$. We will produce a representative $(v^\varepsilon)_\varepsilon$ of the subset of $\mathcal G(\R)$ that gives the notion of ``direction of the movement'' of the mass of $\mu$ when it is deformed in direction $\eta$.
 
 Let $\psi_\varepsilon$ be as in Example \ref{ex:massdeltas}, and $f_\varepsilon=\psi_\varepsilon*\mu$, the convolution. Note that for $0<\varepsilon\ll1$ small enough, we have that $f_\varepsilon\equiv 1$ in a neighborhood of 0 (in fact this is true in most of $U$) and $\psi_\varepsilon\equiv 0$ outside a small neighborhood of 0.
 
 We also let $g_\varepsilon=\psi_\varepsilon*\eta$. For reasons analogous to those outlined in \eqref{eq:convolutionderivative}, we then have that $g_\varepsilon=(-1)^k\psi_\varepsilon^{(k)}$.
 
 In this context, equation \eqref{eq:continuitysmooth} takes the form
 \begin{equation}\label{eq:continuityuniform}
  (-1)^k\psi^{(k)}_\varepsilon=f'_\varepsilon v^\varepsilon+f_\varepsilon (v^\varepsilon)'.
 \end{equation}
 Close to 0, since $f_\varepsilon\equiv 1$, this becomes
 \[
  (-1)^k\psi^{(k)}_\varepsilon=(v^\varepsilon)',
 \]
 so the solution in this region is obviously $v_\varepsilon=(-1)^k\psi_\varepsilon^{(k-1)}$.
 It is easy to see that this solution works globally since $\psi_\varepsilon\equiv 0$ slightly further away from zero. So $(v^\varepsilon)_\varepsilon=((-1)^k\psi^{(k-1)}_\varepsilon)_{0<\varepsilon<1}$ is the representative we were looking for.
\end{example}

 \bibliography{bib}{}
 \bibliographystyle{plain}
\end{document}